\documentclass{article}
\usepackage{amssymb,amsfonts,latexsym,wasysym}

\newtheorem{theorem}{Theorem}[section]
\newtheorem{remark}[theorem]{Remark}
\newtheorem{lemma}[theorem]{Lemma}
\newtheorem{fact}[theorem]{Fact}
\newtheorem{prop}[theorem]{Proposition}
\newtheorem{cor}[theorem]{Corollary}

\newtheorem{ex}[theorem]{Example}

\newtheorem{defi}[theorem]{Definition}

\renewcommand{\c}{\mathfrak{C}}

\newcommand{\tp}{\mathrm{tp}}

\newcommand{\acl}{\mathrm{acl}}

\newcommand{\acle}{\mathrm{acl}^{\mathrm{eq}}}
\newcommand{\dcl}{\mathrm{dcl}}

\renewcommand{\c}{\mathfrak{C}}

\newcommand{\acli}{\mathrm{acl}^{\mathrm{eq}}}

\newcommand{\g}[1]{\mathfrak{#1}}

\def\Ind#1#2{#1\setbox0=\hbox{$#1x$}\kern\wd0\hbox to 0pt{\hss$#1\mid$\hss}
\lower.9\ht0\hbox to 0pt{\hss$#1\smile$\hss}\kern\wd0}
\def\ind{\mathop{\mathpalette\Ind{}}}
\def\Notind#1#2{#1\setbox0=\hbox{$#1x$}\kern\wd0\hbox to 0pt{\mathchardef
\nn=12854\hss$#1\nn$\kern1.4\wd0\hss}\hbox to
0pt{\hss$#1\mid$\hss}\lower.9\ht0 \hbox to
0pt{\hss$#1\smile$\hss}\kern\wd0}

\newenvironment{proof}[1][]{
{\noindent\bf Proof#1}: }{\hfill $\Box$}

\begin{document}
\title{Generic stability and stability}
\author{Hans Adler, Enrique Casanovas and Anand Pillay\thanks{The first and second authors were supported by the Spanish government grant MTM 2011-26840. The  second author was also funded by the Catalan government grant 2009SGR-187.
The third author was supported by EPSRC grant EP/I002284/1.}}
\date{October 22, 2012}
\maketitle

\begin{abstract}
We prove two results about generically stable types $p$ in arbitrary theories.
The first, on existence of strong germs, generalizes results from~\cite{HHM} on stably dominated types.
The second is an equivalence of forking and dividing, assuming generic stability of $p^{(m)}$ for all $m$.
We use the latter result to answer in full generality a question posed by Hasson and Onshuus:
If $p(x)\in S(B)$ is stable and does not fork over $A$ then $p\restriction A$ is stable.
(They had solved some special cases.)
%
\end{abstract}

\section{Stable and generically stable types}
Our notation is standard. We work with an arbitrary complete theory $T$ in
language $L$. $\c$~denotes a `monster model'. $M$ denotes a small elementary
submodel, and $A, B,\dots$ denote small subsets. $L(\c)$ denotes the collection
of formulas with parameters from $\c$, likewise $L(A)$ etc. We sometimes say
$A$-invariant for $Aut(\c/A)$-invariant. By a global type we mean a complete type over $\c$. We assume familiarity with notions from model theory such as heir, coheir, definable type, forking. The book~\cite{Poizat} is a good reference, but see also~\cite{HrushovskiPillay09}.

\subsection*{Stable types}

\begin{defi} \rm Let $\pi(x)$ be a partial type over a set $A$ of parameters.
$\pi$ is \emph{stable} if all complete extensions $p(x)\in S(B)$ of $\pi$ over any set $B\supseteq A$ are definable.
\end{defi}

The definition of \emph{stable partial type} goes back to Lascar and Poizat.
Many other equivalent formulations of this notion are known, which we now mention. These should be considered well known, but  Section 10 of ~\cite{Casanovas} contains proofs and/or references for the  next three Remarks/Facts.

\begin{remark}\label{stable1}  The following are equivalent for any partial type $\pi(x)$ over $A$:
\begin{enumerate}
\item $\pi(x)$ is stable.
\item For  every $B\supseteq A$ there are at most $|B|^{|T|}$ types $p(x)\in S(B)$ extending $\pi(x)$.
\item There are no sequences $(a_i:i<\omega)$ and $(b_i:i<\omega)$ such
that $a_i\models \pi$ for all $i<\omega$ and for some $\varphi(x,y)\in L$,
$\models \varphi(a_i,b_j)\Leftrightarrow i<j$  for all $i,j<\omega$.
\end{enumerate}
\end{remark}


\begin{fact}\label{stable3}
\begin{enumerate}
\item Any extension of a stable type is stable.
\item If $\pi_i(x_i)$ is stable for every $i\in I$, then the partial type $\bigcup_{i\in I}\pi_i(x_i)$ is stable.
\end{enumerate}
\end{fact}

\begin{fact}\label{stable4} (i) Let $p(x)\in S(M)$ be stable. The following are equivalent for any global extension $\g{p}$ of $p$:
\begin{enumerate}
\item $\g{p}$ does not fork over $M$.
\item $\g{p}$ is an heir of $p$.
\item $\g{p}$ is a coheir of $p$.
\item $\g{p}$ is the unique global extension of $p$ which is definable over $M$.
\item $\g{p}$ is $M$-invariant.
\end{enumerate}
(ii) Moreover if $A$ is algebraically closed, and $p(x)\in S(A)$ is stable, then $p$ has a unique global nonforking extension.
\end{fact}

\begin{lemma}\label{acl} If $p(x)\in S(A)$ and some extension of $p$ over $\acl(A)$ is stable, then $p$ is stable.
\end{lemma}
\begin{proof} The assumption implies that all extensions of $p$ over $\acl(A)$ are stable (because they are $A$-conjugate). A counting types argument gives then the result.
\end{proof}

\subsection*{Generically stable types}

Suppose $\g p$ is a global complete type which is~$A$-invariant for some small set~$A$  (in other words for a given  $L$-formula $\phi(x,y)$ whether or not $\phi(x,b)$ is in $\g p$ depends only on $tp(b/A)$). 
A Morley sequence in~$\g p$ over~$A$ then consists of a realization $a_0$ of $\g{p}\restriction A$, a realization $a_1$ of $\g p\restriction Aa_0$, a realization $a_2$ of $\g p\restriction Aa_0a_1$ etc., and one defines $\g{p}^{(n)}\restriction A$ to be $\tp(a_0,\dots,a_{n-1}/A)$. Of course if $A\subseteq B$, then $\g{p}^{(n)}\restriction A \subseteq \g{p}^{(n)}\restriction B$,  and we let  $\g{p}^{n}$ (a complete global type) denote the union as $B$ varies (over small sets). \

The notion of generically stable type originates in works of Shelah, was studied in essentially an $NIP$ environment in~\cite{HrushovskiPillay09}, and in arbitrary first order theories in~\cite{PillayTanovic09}. The following definition comes from the latter paper.

\begin{defi}\rm
 A global type $\g{p}(x)$ is \emph{generically stable over $A$} if it is $A$-invariant and for every ordinal $\alpha\geq \omega$, for every Morley sequence $(a_i:i<\alpha)$ of  $\g{p}$ over $A$ (i.e.,  $a_i\models \g{p}\restriction Aa_{<i}$ for all $i<\alpha)$), the set of all formulas $\varphi(x)$ (over the monster model) satisfied by all but finitely many $a_i$ is a complete type.
\end{defi}


\begin{ex}
There exists a theory $T_{\rm feq2}^*$ which has a generically stable global type $\g p$, invariant over~$\emptyset$,
such that $\g p^{(2)}$ is not generically stable:
There exists a Morley sequence $(a_i : i<\omega)$ of $\g p$ over~$\emptyset$ and a formula $\varphi(x,x',c)$
with parameter~$c$ such that $\models\varphi(a_{2i},a_{2i+1},c)$ holds if and only if $i$ is even.
\end{ex}

\begin{proof}
We should think of $T_{\rm feq2}$ as a two-sorted theory with sorts $P$ (`points') and
$E$ (`equivalence relations') and a single ternary relation that may hold between two elements of~$P$
and one element of $E$, which we will write as $p\sim_eq$.
The axioms of $T_0$ say that each $\sim_e$ is an equivalence relation on~$P$ such that every class
has precisely two elements.
We then add a binary function $f\colon P\times E\to P$ such that each $f_e$ is the involution
on $P$ which fixes every $e$-class but swaps its two elements.
After this modification, the finite models of $T_{\rm feq2}$ have a Fra\"\i ss\'e limit.
Let $T_{\rm feq2}^*$ be its theory. Recall that every theory derived in this way has quantifier elimination.

Let $\g p$ be the generic global type of an element of~$P$.
It is easy to see that $\g p$ is generically stable and definable over~$\emptyset$.
Let $(a_i : i<\omega)$ be a Morley sequence of $\g p$ over~$\emptyset$.
Observe that the sequence of pairs $((a_{2i},a_{2i+1}):i<\omega)$ is a Morley sequence of $\g p^{(2)}$.
By compactness and quantifier elimination we can find an element $e$ of sort~$E$ such that
$a_{2i}\sim_ea_{2i+1}$ if and only if $i$ is even.
\end{proof}

$T_{\rm feq2}^*$ is a variant of Shelah's theory $T_{\rm feq}^*$, the difference being that the equivalence
classes in $T_{\rm feq2}^*$ are unbounded in size. It is interesting to note that in $T_{\rm feq}^*$,
the global generic type of an element of~$P$ is not generically stable.\\

See Section 2 of~\cite{PillayTanovic09} for proofs of the following Remark and Fact, except for Fact 1.9.5 which is due to the third author and appears in~\cite{GarciaOnshuusUsvyatsov10}.

\begin{remark}\label{genstable1} If a global type $\g{p}$ is generically stable
over $A$, then for every $L$-formula $\varphi(x,y)$ there is some $n_\varphi<\omega$
such that for every $\alpha\geq \omega$, for every Morley sequence $(a_i:i<\alpha)$ of  $\g{p}$ over $A$, for every $b$,  if $\{i<\alpha: \models\varphi(a_i,b)\}$ has cardinality $\geq n_\varphi$, then it is cofinite.  Moreover, for every such Morley sequence, $\g{p}$ is the average type of $(a_i:i<\alpha)$, the set of all formulas  satisfied by almost all $a_i$. In particular, every Morley sequence of a generically stable type is totally indiscernible.
\end{remark}

\begin{fact}\label{genstable2}\begin{enumerate}
\item If a global type $\g{p}$ is generically stable over a set~$A$, then it
is definable over~$A$.
\item If a global type $\g{p}$ is generically stable over a set~$B$ and it is $A$-invariant for some set~$A$, then it is generically stable over~$A$.
\item\label{genstable2.3} If a global type $\g{p}$ is generically stable over a set~$A$, then its restriction $\g{p}\restriction A$ is stationary and $\g{p}$ is its only global nonforking extension.
\item If a global type $\g{p}$ is generically stable over a set~$B$ and does not fork over a subset $A\subseteq B$, then it is definable over~$\acli(A)$.
\item\label{genstable2.6} If $\g{p}$ is generically stable over a set~$A$, for any  $a\models \g{p}\restriction A$, for any $b$ such that $\tp(b/A)$ does not fork over~$A$:
$$a\ind_A b \;\;\Leftrightarrow\;\;b\ind_A a$$
(Here the notation $a\ind_A b$ means $tp(a/A,b)$ does not fork over~$A$.)
\end{enumerate}
\end{fact}

\begin{lemma}\label{generic} If a type $p\in S(M)$ is stable and does not fork over $A\subseteq M$, then the unique global nonforking extension of~$p$ is generically stable over~$\acli(A)$.
\end{lemma}
\begin{proof} Let $\g{p}$ be a global extension of $p$ that does not fork over $A$. So $\g{p}$ does not fork over $M$ and by Fact 1.4 coincides with the unique global nonforking extension of~$p$. Stability of~$p$ implies generic stability of $\g{p}$ over $M$. By Fact 1.9.4, $\g{p}$ is generically stable over $\acli(A)$.
\end{proof}

\begin{lemma}\label{genstable3}
\begin{enumerate}
\item Suppose $A\subseteq M$, the global type $\g{p}$ is generically stable over  $A$ (so also over $M$)  and $\g{p}\restriction M$ is stable, then for all $n<\omega$, $\g{p}^{(n)}$ is generically stable over $A$.
\item The same conclusion holds if $\g{p}$ is generically stable over $A$ and $T$ is NIP.
\end{enumerate}
\end{lemma}
\begin{proof}
1. Clearly the type $\g{p}^{(n)}$ is $A$-invariant, and $M$-invariant, and $\g{p}^{(n)}\restriction M$ is stable. Hence
$\g{p}^{(n)}$ does not fork over $A$. So by Lemma 1.10, $\g{p}^{(n)}$ is generically stable over ~$\acli(A)$. By $A$-invariance it is generically stable over $A$.


2. Definability of $\g p$ over~$A$ (which we have by Fact~\ref{genstable2})
is always sufficient for definability of $\g p^{(n)}$ over~$A$.
Note that a Morley sequence of $\g p^{(n)}$ over~$A$ is the same thing as the sequence of successive $n$-tuples in
a Morley sequence of $\g p$ over~$A$. Therefore by Remark~\ref{genstable1} it is totally indiscernible over~$A$.
If $(b_i:i<\kappa)$ is a Morley sequence of $\g p^{(n)}$ such that some formula $\varphi(y,c)$ is satisfied
by neither finitely nor cofinitely many elements of the sequence, then by total indiscernibility
there is a $c'$ such that for $i<\omega$ we have $\models\varphi(b_i,c')$ if and only if $i$~is even.
This contradicts NIP.
\end{proof}

Note that a notion of NIP type suitable for the above lemma can be defined such that every type in a NIP theory and every stable type an arbitrary theory is a NIP type.

%

\section{Main results}
\subsection*{Strong germs}
\nocite{HHM}
We give a proof of the existence of strong germs of definable functions on generically stable types (where the notion is explained below).  This generalizes the same result for the more restricted class of ``stably dominated types", and whose proof in  Chapter 6 of \cite{HHM} was quite involved (using for example a combinatorial lemma from Chapter 5 of the same book). This result played a role in the structural analysis of algebraically closed valued fields in \cite{HHM} as well as later (unpublished) work of Hrushovski on metastable groups. In any case
the third author claimed that there was a simpler proof of the existence of strong germs, assuming only generic stability, and wrote it down for Dugald Macpherson in 2008, under a $NIP$ assumption. In fact the current authors noted that the proof works without any $NIP$ assumption and with the definition of generic stability in Section 1. So it is convenient to include the result and proof in the current paper.

In this section, we will sometimes step outside the monster model~$\c$ and consider realizations of a global type $\g p\in S(\c)$. This is purely for notational convenience, and any tuples and sets not explicitly said to realize a global type will always live in the monster model and be small with relation to it, as usual.

Let $\g p(x)\in S(\c)$ be a global type which is definable over a small set $A\subset\c$. Let $f_c$ be a
partial definable function, defined with a parameter $c$ such that for $a$ realizing $\g p$, $f_c(a)$ is defined.
(We say that $f_c$ is defined on $\g p$.) We define an equivalence relation $E$ on realizations of $tp(c/A)$ in $\c$ (or on a suitable $A$-definable set $Y$ contained in $tp(c/A)$): $E(c_1,c_2)$ if  $f_{c_1}(a) = f_{c_2}(a)$ for some (any) $a$ realizing $\g p$.

By definability of $\g p$, $E$ is $A$-definable, and $c/E$ is called the germ of $f_c$ on $\g p$.
We say that this germ is \emph{strong} over~$A$,
if for $a$ realizing $\g p$  (or realizing $\g{p}\restriction (A,c)$), $f_c(a)\in\dcl(c/E, a, A)$. Note that this amounts to their being a $c/E$-definable function $g_{c/E}(-)$ such for  $a$ realizing $\g{p}\restriction (A,c)$, $f_{c}(a) = g_{c/E}(a)$. So the function $f_{c}$ is definable over its germ.


As any global type $\g p\in S(\c)$ that is generically stable over a set~$A$ is definable over~$A$, we can speak of germs of definable functions on $\g p$.

%
%
%
%

The following Lemma is key for the proof of Theorem~\ref{thmgerm}.
\begin{lemma}\label{lemgerm} Suppose the global type $\g q(x,y) = tp(a,b/\c)$ is generically stable over a set~$A\subset\c$.
\begin{enumerate}
\item If $b\in\acl(a,\c)$, then there is $n <\omega$ such that for any Morley sequence $((a_i,b_i):i<\omega)$
in $\g q$ over $A$, we have $b_n\in\acl(A, a_0,b_0,\dots,a_{n-1},b_{n-1},a_n)$.
\item If $b\in\acl(a,\c)$, then $b\in\acl(a,A)$.
\item\label{lemgerm.3} If $b\in\dcl(a,\c)$, then $b\in\dcl(a,A)$.
\end{enumerate}
\end{lemma}

\begin{proof}
1. If the claim is false, then for any $c\in\c$ we can inductively build a Morley sequence
$((a_i,b_i):i<\omega)$ in $\g q$ over~$A$
such that $b_i\notin\acl(a_ic)$ for all $i<\omega$.
Pick a formula $\varphi(x,y,c)\in\g q(x,y)$ which witnesses $b\in\acl(a,\c)$ in the sense that
for some $k$ we have $\c\models\forall x\exists_{<k}y\varphi(x,y,c)$.
It follows that $\models\neg\varphi(a_i,b_i,c)$ for all $i<\omega$,
contradicting Remark~\ref{genstable1}.

%
%
%

2. Fix now $((a_i,b_i):i<\omega)$ a Morley sequence in $\g q$ over~$A$.
Let $d = (a_0,b_0,\dots,a_{n-1},b_{n-1})$.
So $d$ is a realization of $\g q^{(n)}\restriction A$, and by total indiscernibility,
$d$ also realizes $\g q^{(n)}\restriction (Aa_nb_n)$.
Now $\g q^{(n)}$ is definable over~$A$. In particular $tp(d/Aa_nb_n)$ does not divide over~$A$.
We now show that $b_n\in\acl(a_n,A)$.
If not, we can find an infinite indiscernible over~$Aa_n$ sequence $(b_n^j:j<\omega)$
with $b_n^0 = b_n$ and the $b_n^j$'s all different.
Then $((a_n,b_n^j):j<\omega)$ is an indiscernible sequence over~$A$.
Hence there is $d'$ such that $tp(d'a_nb_n^j/A) = tp(da_nb_n/A)$ for all $j$.
But by 1, $b_n^0 = b_n\in\acl(A,d',a_{n})$ and we get a contradiction.

3. By point \ref{genstable2.3} of Fact~\ref{genstable2},
$\g q$ is the unique nonforking extension of $\g q\restriction A$.
Suppose that $b\in\dcl(\c,a)$, but $b\notin\dcl(A,a)$.
So by 2, $b\in\acl(A,a)$, and there is $b'\neq b$ with the same type as $b$ over $Aa$.
In particular $b'\in\acl(A,a)$.  As $tp(a/\c)$ is definable over $A$,
it follows that $\g q' = tp(a,b'/\c)$ does not fork over $A$.
But then $\g q'$ and $\g q$ are distinct global nonforking extensions of $\g q\restriction A$
(for a certain $\c$-definable function $f$, $b = f(a)$, but $b'\neq f(a)$), a contradiction.
\end{proof}

\begin{theorem}\label{thmgerm} Suppose that the global type $\g p\in S(\c)$ is generically stable over a
small set~$A\subset\c$, and $f_c$ is a definable function, defined on $\g p$.
Then the germ of $f_c$ on $\g p$ is strong over~$A$.
\end{theorem}

\begin{proof}
For notational simplicity we assume $A=\emptyset$.
Let $a$ realize $\g p$  (in of course a model containing~$\c$), and $b = f_c(a)$.
Let $\g q = tp(a,b/\c)$. As $\g p$ is generically stable over~$c$ and $b\in\dcl(a,c)$, clearly
$\g q$ is also generically stable over~$c$.

Let $c/E$ be the germ of $f_c$ at $\g p$.
Towards proving that $\g q$ is actually $c/E$-invariant, let $h$ be any automorphism of $\c$ fixing $c/E$.
Extend to an automorphism $h'$ of the surrounding monster model containing $(a,b)$.
As $tp(a/\c) = tp(h'(a)/\c)$ we may assume $a = h'(a)$. As $E(c,h(c))$, $h'(b) = f_{h(c)}(a) = f_c(a) = b$.
So $h'(a,b) = (a,b)$.
Hence $\g q$ is fixed by $h$, and so $\g q$ is $c/E$-invariant, hence definable over~$c/E$.

As $\g q$ is definable over $c/E$ and $b\in\dcl(a,c)\subseteq\dcl(a,\c)$, by Lemma~\ref{lemgerm}
we have $b\in\dcl(a, c/E)$, i.e., the germ of $f_c$ is strong over~$A=\emptyset$.
\end{proof}

\subsection*{Forking, dividing, and preservation of stability}

\begin{lemma}\label{genstable4} Let  $\g{p}(x)\in S(\c)$ and assume that for all $n<\omega$, $\g{p}^{(n)}$  is generically stable over $A$. For any $\varphi(x,y)\in L$, for any tuple $m$, the following are equivalent:
\begin{enumerate}
\item  $\varphi(x,m)\in \g{p}(x)$
\item $(\g{p}(x)\restriction A) \cup\{\varphi(x,m)\}$ does not fork over $A$.
\item $(\g{p}(x)\restriction A) \cup\{\varphi(x,m)\}$ does not divide over $A$.
\end{enumerate}
\end{lemma}
\begin{proof}  It is clear that \emph{1} $\Rightarrow$ \emph{2} $\Rightarrow$
\emph{3}. We prove \emph{3} $\Rightarrow$ \emph{1}. Let $\varphi(x,m)\notin \g{p}$.
We will assume that $(\g{p}(x)\restriction A)\cup \{\varphi(x,m)\}$ does not divide over~$A$ and aim for a contradiction. By generic stability of $\g{p}$, there is a greatest natural number $N$ such that for some Morley sequence $(a_i:i<\omega)$ of $\g{p}$ over $A$,
$\models\varphi(a_i,m)$ for $N$ many $a_i$.
Fix such a Morley sequence $(a_i:i<\omega)$, and without loss of generality we have $\models\varphi(a_i,m)$ for all $i<N$.
Let $b_0 = (a_0,\ldots,a_{N-1})$, $b_1 = (a_N,\ldots,a_{2N-1})$, etc., and let $m_0 = m$.
So $(b_i:i<\omega)$ is a Morley sequence in $\g p ^{(N)}$ over~$A$.
We can find $m_1, m_2,\ldots$ such that $(b_im_i:i<\omega)$ is indiscernible over~$A$.
We can extend this sequence to $(b_im_i:i<\kappa)$ for any $\kappa$.
Now our assumption that $(\g p(x)\restriction A)\cup \{\varphi(x,m)\}$ does not divide over~$A$
implies that there is $a$ realizing $(\g{p}(x)\restriction A)\cup \{\varphi(x,m_{i}):i<\kappa\}$.
\newline
{\em Claim.}  For some $i<\kappa$, $b_i$ is independent from $a$ over~$A$
(i.e. $b_i$ realizes $\g p^{(N)}\restriction Aa$).
\newline
{\em Proof of claim.}  By generic stability of $\g p^{(N)}$, we have that for every formula $\psi(y,x)\in L(A)$ such that
$\psi(y,a)\notin\g p^{(N)}$ there are only finitely many $b_i$ such that $\models\psi(b_i,a)$.
Hence for $\kappa$ large enough, there is $b_i$ such that $\models \neg\psi(b_i,a)$ for each formula $\psi(y,a)\notin \g p^{(N)}$.
Namely $b_i$ realizes $\g p^{(N)}\restriction Aa$ as required.

By the claim and point 5 of Fact~\ref{genstable2}, $a$ is independent of $b_i$
over~$A$, so $(b_i,a)$ extends to an infinite Morley sequence of $p$ over $A$ for which there are at least $N+1$ elements satisfying $\varphi(x,m_i)$. This is a contradiction with our choice of $N$  (as $\tp(m_i/A) = \tp(m/A)$).
\end{proof}

\begin{prop}\label{genstable5} Assume $\g p^{(n)}(x)$ is  generically stable over $A$ for every $n<\omega$. Let $M$ be a model containing $A$.
Let $a_0,a_1,\ldots$ be realizations of $\g{p}\restriction A$ and let $\Sigma(x_0 x_1\ldots) = \tp(a_0, a_1,\ldots/A)$. Then there is a realization
$(a_0^\prime,a_1^\prime,\ldots)$ of $\Sigma$ such that each $a_i^\prime$ realizes $\g p\restriction M$.
\end{prop}

\begin{proof}
We may assume $M$ is $\omega$-saturated. It is enough to prove that
\[ \Sigma(x_0 x_1,\ldots)\cup\big\{\neg\varphi(x_i,m) : i<\omega,\,\varphi(x,z)\in L,\,
  m\in M,\,\varphi(x,m)\notin\g p(x)\restriction M\big\}\]
is consistent. Towards a contradiction, suppose not. Then it is not hard to see that
\[\Sigma(x_0 x_1,\ldots) \models \bigvee_{i<n}\varphi(x_i,m)\]
for some number $n$ and formula $\varphi(x,m)\in L(M)\setminus \g p(x)$.
By Lemma~\ref{genstable4}, $(\g p(x)\restriction A) \cup \{\varphi(x,m)\}$ divides over~$A$.
This is witnessed by a sequence $(m_j:j<\omega)$ of $A$-indiscernibles such that
$(\g p(x)\restriction A) \cup \{\varphi(x,m_j):j<\omega\}$ is inconsistent. For each $j$ we have
\[\Sigma(x_0 x_1,\ldots) \models \bigvee_{i<n}\varphi(x_i,m_j),\]
so for each $j$ there is an $i<n$ such that $\models\varphi(a_i,m_j)$.
By the pigeonhole principle there is an i such that for infinitely many $j$ we have $\models\varphi(a_i,m_j)$,
contradicting inconsistency of $(\g p(x)\restriction A) \cup \{\varphi(x,m_j):j<\omega\}$.
\end{proof}

\begin{cor}\label{preservation}
Let $p(x)\in S(B)$ be a stable type. If  $p$ does not fork over $A\subseteq B$, then $p\restriction A$ is stable.
\end{cor}
\begin{proof}
Without loss of generality, let $B$ be a model.
Using Lemma~\ref{generic}, we can obtain a global type $\g p\supset p$ generically stable over~$\acle A$.
Note that by Lemmas~\ref{acl} and~\ref{genstable3}, each $\g{p}^{(n)}$ is generically stable over~$A$ and over~$B$.
Assume $p\restriction A$ is unstable. By Remark~\ref{stable1}, there exist a sequence $(a_i:i<\omega)$ of realizations of $p\restriction A$, a sequence $(b_j:j<\omega)$, and a formula $\varphi(x,y)$ such that $\models\varphi(a_i,b_j)\Leftrightarrow i<j$.
By Proposition~\ref{genstable5}, there is a sequence $(a'_i:i<\omega)$ with the same type as $(a_i:i<\omega)$ over~$A$ and such that each $a'_i$ realizes~$p$. The new sequence, together with a corresponding sequence $(b'_j:j<\omega)$ and $\varphi(x,y)$, witnesses that $p$ is unstable by Remark~\ref{stable1}.
\end{proof}

This result generalizes a similar claim in~\cite{HassonOnshuus10}, where it is assumed that $T$ is NIP, and provides a correct proof.


\noindent{\sc
University of Vienna}\\
{\tt johannes.aquila@googlemail.com}\\

\noindent{\sc
University of Barcelona\\
{\tt e.casanovas@ub.edu}\\

\noindent{\sc University of Leeds}\\ {\tt A.Pillay@leeds.ac.uk}


\end{document}